\documentclass[11pt]{article}
\usepackage[dvips]{graphicx}
\usepackage{color}
\usepackage{amsmath}
\usepackage{amsfonts}
\usepackage{amssymb}
\usepackage{amsthm}
\usepackage{newlfont}
\usepackage{epsfig}
\usepackage{multirow}
\usepackage{setspace}
\usepackage{hyperref}
\usepackage{cite}

\begin{document}

\newtheorem{assumption}{Assumption}[section]
\newtheorem{definition}{Definition}[section]
\newtheorem{lemma}{Lemma}[section]
\newtheorem{proposition}{Proposition}[section]
\newtheorem{theorem}{Theorem}[section]
\newtheorem{corollary}{Corollary}[section]
\newtheorem{remark}{Remark}[section]

\title{Linear Conjugacy of Chemical Reaction Networks}
\author{Matthew D. Johnston\thanks{Supported by a Natural Sciences and Engineering Research Council of Canada Post-Graduate Scholarship} \vspace*{.2in} and David Siegel\thanks{Supported by a Natural Sciences and Engineering Research Council of Canada Discovery Grant \newline \textbf{Keywords:} chemical kinetics; stability theory; persistence; complex balancing; dynamical equivalence \newline \textbf{AMS Subject Classifications:} 80A30, 34D20, 37C75.} \\
Department of Applied Mathematics, University of Waterloo, \\
Waterloo, Ontario, Canada N2L 3G1}
\date{}
\maketitle

\tableofcontents

\bigskip

\begin{abstract}

Under suitable assumptions, the dynamic behaviour of a chemical reaction network is governed by an autonomous set of polynomial ordinary differential equations over continuous variables representing the concentrations of the reactant species. It is known that two networks may possess the same governing mass-action dynamics despite disparate network structure. To date, however, there has only been limited work exploiting this phenomenon even for the cases where one network possesses known dynamics while the other does not. In this paper, we bring these known results into a broader unified theory which we call conjugate chemical reaction network theory. We present a theorem which gives conditions under which two networks with different governing mass-action dynamics may exhibit the same qualitative dynamics and use it to extend the scope of the well-known theory of weakly reversible systems.

\end{abstract}

\bigskip

\section{Introduction}

A chemical reaction network is given by sets of reactant species reacting at prescribed rates to form sets of product species. Under appropriate assumptions (well-mixing, constant external conditions, large number of reacting molecules, mass-action kinetics, etc.) such systems can be modeled deterministically by an autonomous set of ordinary differential equations over continuous variables representing the concentrations of the reactant species. The resulting mathematical models have a long history and enjoy applications in fields such as systems biology, industrial chemistry, pharmaceutics, etc. \cite{B-B1,F2,P-S,S}. Active theoretical topics within the mathematical literature include determining network conditions which guarantee locally stable dynamics \cite{F1,H,H-J1,S-C,S-M,V-H}, conditions which prohibit/permit multistability \cite{C-F1,C-F2,C-P2,S-F}, and conditions which ensure the persistence of the chemical species \cite{A,A-S,C-D-S-S,C-N-P,S}.

One topic which has gained only fleeting attention until recently has been determining conditions under which networks with different reaction structure exhibit the same qualitative dynamics. The most comprehensive of these results is Craciun and Pantea \cite{C-P} wherein the authors consider conditions under which a mass-action system does not permit uniquely identifiable rate constants in the sense that two disparate reaction networks could produce the same governing set of differential equations under the assumption of mass-action kinetics for specified choices of the rate constants. Within the literature, this is sometimes referred to as the ``fundamental dogma of chemical kinetics'' \cite{C-P-R,Sz2,Sz-H}. This work has recently been continued by G. Szederk\'{e}nyi \emph{et al.} who have presented computational methods for determining networks which have equivalent dynamics \cite{Sz2,Sz-H,Sz-H-P}. Related work in the literature can be found in \cite{A1,C-P-R,H-J1,K,M}.

In this paper, we attempt to bring these results together into a unified framework and language. We have chosen to borrow from dynamical systems theory in calling two mass-action systems with related dynamics \emph{conjugate} systems \cite{P,W}. More specifically, we will call two mass-action systems conjugate if there is a mapping which takes trajectories of one system into trajectories of the other. We will focus on the subset of \emph{linearly conjugate} systems. To the best of our knowledge, all of the results contained in the literature to date on qualitative equivalence of mass-action systems restrict their focus to conditions under which the governing differential equations for two networks are identical; consequently, the required conjugacy mapping is the identity. We will go further than these results with Theorem \ref{maintheorem} in giving conditions on the network for which a non-trivial linear mapping is required to demonstrate conjugacy.

In practice, we are typically presented with a specified network and asked to determine the dynamic behaviour. Linear conjugate chemical reaction network theory presents a powerful approach to solving such a problem: if a network with unknown dynamics can be shown to be linearly conjugate to a network with known dynamics, the dynamical properties of the second apply to the first as well. While the second network is generally unknown, the apparent dynamics of the first network may provide clues as to the class of networks within which to search for a linearly conjugate network. We focus particularly on the class of weakly reversible networks considered in \cite{F1,H,H-J1}. Szederk\'{e}nyi and Hangos have considered some aspects of this problem from a computational perspective \cite{Sz-H}.


The paper is organized as follows: in Section \ref{backgroundsection} we introduce the terminology and definitions which will be used throughout the paper; in Section \ref{conjugatedynamicalsystemssection} we introduce our notion of conjugate chemical reaction networks, present the relevant results from the literature, give a new result (Theorem \ref{maintheorem}) which goes beyond this known theory, and connect our results to weakly reversible networks through examples; we close in Section \ref{conclusionsection} with some reflections on the results obtained and avenues for future work.

Throughout the paper, we will let $\mathbb{R}_{>0}^m$ and $\mathbb{R}_{\geq 0}^m$ denote the $m$-dimensional spaces with all coordinates strictly positive and non-negative, respectively.

\section{Background}
\label{backgroundsection}

In this section, we introduce the concepts and notation relevant to the study of chemical reaction networks. We also summarize the work conducted to date on the subject of conjugacy of chemical reaction networks.

\subsection{Chemical Reaction Networks}

We will let $\mathcal{A}_j$ denote the \emph{species} or \emph{reactants} of the network and define $|\mathcal{S}| = m$ where $\mathcal{S}$ is the set of distinct species of the network. A chemical reaction network consists of the union of the elementary reaction set
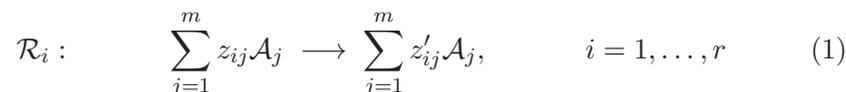
\begin{equation}
\label{crn}
\mathcal{R}_i: \hspace{0.5in} \sum_{j=1}^m z_{ij} \mathcal{A}_j \; \longrightarrow \; \sum_{j=1}^m z_{ij}' \mathcal{A}_j, \hspace{0.5in} i=1, \ldots, r
\end{equation}
where $z_{ij}, z_{ij}' \in \mathbb{Z}_{\geq 0}$ are the \emph{stoichiometric coefficients} of the $i^{th}$ reaction. The set of all reactions in the network will be denoted $\mathcal{R}$ so that $|\mathcal{R}|=r$. It will occasionally be convenient to adopt the short-hand $\mathcal{C}_i = \sum_{j=1}^m z_{ij} \mathcal{A}_j$ and $\mathcal{C}_i' = \sum_{j=1}^m z_{ij}' \mathcal{A}_j$ so that the reaction network can be given by
\begin{equation}
\label{crn2}
\mathcal{R}_i: \hspace{0.5in} \mathcal{C}_i \; \longrightarrow \; \mathcal{C}_i' \hspace{0.5in} i=1, \ldots, r.
\end{equation}
In this setting, $\mathcal{C}_i$ will be referred to as the \emph{reactant complex} and $\mathcal{C}_i'$ will be referred to as the \emph{product complex}. A chemical reaction network will be denoted by the triplet $\mathcal{N} = (\mathcal{S},\mathcal{C},\mathcal{R})$.

There is no demand that the stoichiometry of each reactant and product complex be distinct except that $\mathcal{C}_i \not= \mathcal{C}_i'$ for $i=1, \ldots, r$ (so each reaction is meaningful). That is to say, a single stoichiometrically distinct complex may be the reactant or product complex for numerous reactions. It will occasionally be convenient to index the complexes not according to the reactions they occur in, but by their stoichiometric distinctiveness. In this setting, reactions are not represented as a list but as interactions between the $n$ stoichiometrically distinct complexes of the network. We will denote by $\mathcal{C}^{i}$, $i = 1, \ldots, n$, the $n$ stoichiometrically distinct complexes of the network and let $\mathcal{C}$ denote the set of such complexes.

We are often interested in the algebraic structure of a reaction network. These properties will be most readily seen by indexing the complexes by their stoichiometric distinctiveness according to $\mathcal{C}^i$, $i = 1, \ldots, n$. The following concepts will be of interest to us.


\begin{definition}
\label{path}
Consider a chemical reaction network $\mathcal{N}$. We will say there is a \textbf{path} from $\mathcal{C}^i \in \mathcal{C}$ to $\mathcal{C}^j \in \mathcal{C}$ if there exists a sequence of stoichiometrically distinct complexes $\left\{ \mathcal{C}^{\nu_1} \right.$, $\mathcal{C}^{\nu_2},$ $\ldots,$ $\left. \mathcal{C}^{\nu_l} \right\},$ such that $\mathcal{C}^{i} = \mathcal{C}^{\nu_1}$, $\mathcal{C}^{j}=\mathcal{C}^{\nu_l}$, and
\[\mathcal{C}^{\nu_k} \longrightarrow \mathcal{C}^{\nu_{k+1}} \in \mathcal{R}\]
for every $k=1, \ldots, l-1$.
\end{definition}




\begin{definition}
\label{weaklyreversible}
A chemical reaction network $\mathcal{N}$ is said to be \textbf{weakly reversible} if the existence of a path from $\mathcal{C}^{i}$ to $\mathcal{C}^{j}$ implies the existence of a path from $\mathcal{C}^{j}$ to $\mathcal{C}^{i}$.
\end{definition}

Intuitively, weak reversibility means every reaction can be reversed by a suitable sequence of subsequent reactions. It is not necessary that every reaction have a reverse reaction (i.e. that the system be reversible), as can be seen by the weakly reversible network
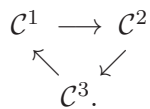
\[\begin{array}{c} \mathcal{C}^1 \; \longrightarrow \; \mathcal{C}^2 \\ \nwarrow \hspace{0.2in} \swarrow \\ \mathcal{C}^3.\end{array}\]

\subsection{Mass-Action Kinetics}

We are particularly interested in the evolution of the concentrations of the chemical species as time progresses. We will let $x_i = [ \mathcal{A}_i ]$ denote the concentration of the $i^{th}$ species and denote by $\mathbf{x} = [x_1 \; x_2 \; \cdots \; x_m]^T \in \mathbb{R}_{\geq 0}^m$ the \emph{concentration vector}.

The differential equations governing the chemical reaction network $\mathcal{N}$ given by (\ref{crn}) under the assumption of mass-action kinetics are given by
\begin{equation}
\label{de}
\frac{d\mathbf{x}}{dt} = \mathbf{f}(\mathbf{x}) = \sum_{i = 1}^r k_i \: ( \mathbf{z}_i' - \mathbf{z}_i ) \: \mathbf{x}^{\mathbf{z}_i}
\end{equation} 
\noindent where $\mathbf{x}^{\mathbf{z}_{i}} = \prod_{j=1}^m x_j^{z_{ij}}$ and $k_i>0$ is the \emph{rate constant} for the $i^{th}$ reaction. We will denote by $\mathbf{k} = [ k_1 \; k_2 \; \cdots \; k_r]^T \in \mathbb{R}_{> 0}^r$ the \emph{rate constant vector}.

Chemical reaction networks endowed with mass-action kinetics (\ref{de}) will be called \emph{mass-action systems} and will be denoted by $(\mathcal{S},\mathcal{C},\mathcal{R},\mathbf{k})$. Since mass-action kinetics is the only form of kinetics considered in this paper, we will use $\mathcal{N}$ interchangeably to denote both the reaction network $(\mathcal{S},\mathcal{C},\mathcal{R})$ and the mass-action system $(\mathcal{S},\mathcal{C},\mathcal{R},\mathbf{k})$.

Several fundamental properties of mass-action systems are readily seen from (\ref{de}). In particular, it is clear that solutions are not able to wander around freely in $\mathbb{R}_{>0}^m$. The following concepts help clarify where solutions may lie.

\begin{definition}
\label{kinetic}
The \textbf{kinetic subspace} for a mass-action system is the smallest linear subspace $S^* \subset \mathbb{R}^m$ such that
\[S^* = \mbox{im}\left(\mathbf{f}(\mathbf{x})\right).\]
\noindent The dimension of the kinetic subspace will be denoted by $| S^* | = s^*$.
\end{definition}

\begin{definition}
\label{stoic}
The \textbf{stoichiometric subspace} for a mass-action system is the linear subspace $S \subset \mathbb{R}^m$ such that
\[S= \mbox{span} \left\{ \left. (\mathbf{z}_i'-\mathbf{z}_i) \; \right| \; i = 1, \ldots, r \right\}.\]
\noindent The dimension of the stoichiometric subspace will be denoted by $| S | = s$.
\end{definition}

While the stoichiometric and kinetic subspaces are related, they need not coincide, in which case we have the strict inclusion $S^* \subset S$ (see Example 2 in Section \ref{weaklyreversiblenetworkssection}). Since we clearly have $\mathbf{f}(\mathbf{x}) \in S$ and $\mathbf{f}(\mathbf{x}) \in S^*$, by integrating (\ref{de}) we can see that solutions can be partitioned according to their initial conditions $\mathbf{x}_0 \in \mathbb{R}_{>0}^m$ in the following way.

\begin{proposition}[\cite{H-J1,V-H}]
\label{proposition2}
Let $\mathbf{x}(t)$ be the solution to (\ref{de}) with $\mathbf{x}(0)=\mathbf{x}_0 \in \mathbb{R}^m_{>0}$. Then $\mathbf{x}(t) \in \mathsf{C}^*_{\mathbf{x}_0} = (\mathbf{x}_0 + S^*) \cap \mathbb{R}^m_{>0}$ for $t \geq 0$ where $\mathsf{C}^*_{\mathbf{x}_0}$ is called the positive \textbf{kinetic compatibility class} associated with $\mathbf{x}_0 \in \mathbb{R}^m_{>0}$.
\end{proposition}

In other words, the sets $\mathsf{C}^*_{\mathbf{x}_0}$ are invariant spaces of the dynamics (\ref{de}). A more commonly encountered invariant space of (\ref{de}) is the positive \emph{stoichiometic compatibility class} $\mathsf{C}_{\mathbf{x}_0} = (\mathbf{x}_0 + S) \cap \mathbb{R}^m_{>0}$. The sets $\mathsf{C}_{\mathbf{x}_0}$ and $\mathsf{C}^*_{\mathbf{x}_0}$ are related by the following result, which allows us to restrict attention to $\mathsf{C}_{\mathbf{x}_0}$ for weakly reversible networks.

\begin{lemma}[Corollary 1, \cite{F-H}]
\label{kineticstoichiometric}
For weakly reversible mass-action systems the stoichiometric and kinetic subspaces coincide.
\end{lemma}




\subsection{Complex Balanced Systems}
\label{detailedandcomplexbalancedsystemssection}

Many mass-action systems can be classified according to the nature of the equilibrium concentrations they permit. A particularly important class of such systems, which will be considered extensively in this paper, are \emph{complex balanced} systems (see \cite{C-D-S-S,H-J1,V-H}).


\begin{definition}
\label{complexbalanced}
An equilibrium concentration $\mathbf{x}^* \in \mathbb{R}_{>0}^m$ of a mass-action system (\ref{de}) will be called \textbf{complex balanced} if, for every stoichiometrically distinct complex $\mathcal{C}^0 \in \mathcal{C}$, we have
\[\mathop{\sum_{j=1}^r}_{\mathcal{C}_j'=\mathcal{C}^0} k_j \: (\mathbf{x}^*)^{\mathbf{z}_j} = (\mathbf{x}^*)^{\mathbf{z}_i} \mathop{\sum_{j=1}^r}_{\mathcal{C}_j=\mathcal{C}^0} k_j.\]
A mass-action system will be called a \textbf{complex balanced system} if every equilibrium concentration is a complex balanced equilibrium concentration.
\end{definition}
In other words, an equilibrium concentration is complex balanced if the net flow rate into every stoichiometrically distinct complex is exactly balanced by the net flow rate out of the complex. It is known that every mass-action system permitting a complex balanced equilibrium concentration corresponds to a weakly reversible reaction network (Theorem 2B, \cite{H}).

Analysis of complex balanced systems is made easier by the following lemma.


\begin{lemma}[Lemma 5B,\cite{H-J1}]
\label{lemma1}
If a mass-action system is complex balanced at some concentration $\mathbf{x}^* \in \mathbb{R}_{>0}^m$, then it is complex balanced at all equilibrium concentrations.
\end{lemma}
\noindent Consequently, any mass-action system permitting a complex balanced equilibrium concentration is a complex balanced system. It is worth noticing, however, that a mass-action system may be complex balanced for some values of the rate constants and not for others.

\subsection{Locally Stable Dynamics}
\label{locallystabledynamicssection}

Mass-action systems permit a variety of behaviours, including multistability, periodic behaviour, and chaotic behaviour (see \cite{E-T}). In this paper, we will be particularly interested in networks exhibiting \emph{locally stable dynamics}.

\begin{definition}
\label{locallystabledynamics}
A mass-action system is said to exhibit \textbf{locally stable dynamics} if there is a unique positive equilibrium concentration $\mathbf{x}^* \in \mathbb{R}_{>0}^m$ within each kinetic compatibility class $\mathsf{C}^*_{\mathbf{x}_0}$ and that equilibrium concentration is locally asymptotically stable relative to $\mathsf{C}^*_{\mathbf{x}_0}$.
\end{definition}

Locally stable dynamics is a particularly desirable form of dynamics since the behaviour is so regular and predictable; however, it is local in nature and does not preclude trajectories from approaching $\partial \mathsf{C}^*_{\mathbf{x}_0}$. Mass-action systems for which each $\mathbf{x}^*$ is a global attractor for $\mathsf{C}^*_{\mathbf{x}_0}$ are called \emph{globally stable}. The following result was derived by Horn and Jackson in \cite{H-J1}.
\begin{theorem}[Theorem 6A and Lemma 4C, \cite{H-J1}]
\label{hornjackson}
If a mass-action system  is complex balanced then it possesses locally stable dynamics.
\end{theorem}
\noindent Since complex balanced systems are necessarily weakly reversible, it follows by Lemma \ref{kineticstoichiometric} that we may restrict our attention to the positive stoichiometric compatibility classes $\mathsf{C}_{\mathbf{x}_0}$ in the definition of locally stable dynamics for such systems.

This result is powerful in that it guarantees a particularly desirable form of dynamics based solely on consideration of the equilibrium set of the system. Together with Feinberg, the authors go further in relating the complex balancing condition to the algebraic structure of the reaction network (\ref{crn2}). In particular, they demonstrate that every weakly reversible mass-action system is complex balanced for at least some choice of the rate constants \cite{F1,H}. They also show that there are algebraic conditions on a weakly reversible network (\ref{crn2}) sufficient to guarantee the associated mass-action system is complex balanced for all rate constants (Theorem 4A, \cite{H}). Further consideration of what the conditions on the rate constants required to guarantee complex balancing look like is contained in \cite{C-D-S-S}.

\section{Conjugate Chemical Reaction Networks}
\label{conjugatedynamicalsystemssection}

The focus of this paper is determining when two mass-action systems can be shown to have qualitatively identical behaviour despite disparate reaction networks (\ref{crn}). Our approach to this problem is to show that there is a suitably nice mapping between the flows of (\ref{de}) for one network and another. In the standard theory of differential equations, this notion is captured in the well-studied concept of \emph{conjugacy} (for a complete introduction to conjugacy of dynamical systems, see \cite{P,W}).

We define the reaction network $\mathcal{N}'=(\mathcal{S},\mathcal{C}',\mathcal{R}')$ as consisting of the reaction set
\begin{equation}
\label{crn3}
\mathcal{R}_{i}': \hspace{0.5in} \sum_{j=1}^m \tilde{z}_{ij} \mathcal{A}_j \; \longrightarrow \; \sum_{j=1}^m \tilde{z}_{ij}' \mathcal{A}_j, \hspace{0.5in} i=1, \ldots, \tilde{r}
\end{equation}
or, alternatively,
\begin{equation}
\label{crn4}
\mathcal{R}_{i}': \hspace{0.5in} \tilde{\mathcal{C}}_i \; \longrightarrow \; \tilde{\mathcal{C}}_i', \hspace{0.5in} i=1, \ldots, \tilde{r}.
\end{equation}
The associated mass-action system will be denoted $(\mathcal{S},\mathcal{C}',\mathcal{R}',\mathbf{k}')$ where the rate constants are given by $\tilde{k}_i > 0$, $i = 1, \ldots, \tilde{r}$. We will let $\Phi(\mathbf{x}_0,t)$ denote the flow associated with the mass-action kinetics (\ref{de}) for $\mathcal{N}$ and $\Psi(\mathbf{y}_0,t)$ denote the flow associated with the mass-action kinetics (\ref{de}) for $\mathcal{N}'$. We will adopt the convention of referring to $\mathcal{N}$ as the \emph{original network} and $\mathcal{N}'$ as the \emph{target network}.

Note that, while we follow the notation of \cite{C-P} in denoting any second network by $\mathcal{N}'=(\mathcal{S},\mathcal{C}',\mathcal{R}')$, we distinguish the relevant components of the second network using tild\'{e}s to avoid confusion with the vectors $\mathbf{z}_i'$ from the first system. Also notice that the networks $\mathcal{N}$ and $\mathcal{N}'$ are allowed to have not only different complexes and reactions, but different \emph{numbers} of complexes and reactions; the number of species $|\mathcal{S}|=m$, however, is required to be the same.

We are now prepared to define our notion of conjugacy of chemical reaction networks.

\begin{definition}
\label{conjugate}
Consider two mass-action systems $\mathcal{N}$ and $\mathcal{N}'$. We will say $\mathcal{N}$ and $\mathcal{N}'$ are \textbf{$\mathbf{C}^k$-conjugate} if there exists a $\mathbf{C}^k$-diffeomorphism $\mathbf{h}: \mathbb{R}^m_{>0} \mapsto \mathbb{R}_{>0}^m$ such that $\mathbf{h}(\Phi(\mathbf{x}_0,t))=\Psi(\mathbf{h}(\mathbf{x}_0),t)$ for all $\mathbf{x}_0 \in \mathbb{R}_{>0}^m$.
\end{definition}

\begin{definition}
\label{linearconjugate}
We will say $\mathcal{N}$ and $\mathcal{N}'$ are \textbf{linearly conjugate} if they are $\mathbf{C}^\infty$-conjugate and the diffeomorphism $\mathbf{h}: \mathbb{R}_{>0}^m \mapsto \mathbb{R}_{>0}^m$ is linear.
\end{definition}

In this paper, we will focus on the notion of linear conjugacy. Note that any linear diffeomorphism is necessarily $\mathbf{C}^\infty$ so that any linear conjugacy is a $\mathbf{C}^\infty$-conjugacy. The following results clarify the form the linear mapping $\mathbf{h}: \mathbb{R}_{>0}^m \mapsto \mathbb{R}_{>0}^m$ may take and the implications of conjugacy.

\begin{lemma}
\label{sillylemma}
A linear, bijective mapping $\mathbf{h}: \mathbb{R}_{>0}^m \mapsto \mathbb{R}_{>0}^m$ may consist of at most positively scaling and reindexing of coordinates.
\end{lemma}

\begin{proof}
Consider a linear, bijective mapping $\mathbf{h}: \mathbb{R}_{>0}^m \mapsto \mathbb{R}_{>0}^m$. Since $\mathbf{h}(\mathbf{x})$ is linear, it can be represented $\mathbf{h}(\mathbf{x})=A\mathbf{x}$ where $A \in \mathbb{R}^{m \times m}$ and since $\mathbf{h}(\mathbf{x})$ is bijective, it has an inverse $\mathbf{h}^{-1}(\mathbf{x})=A^{-1}\mathbf{x}.$ Since the mappings are from $\mathbb{R}_{>0}^m$ to $\mathbb{R}_{>0}^m$, all entries in $A$ and $A^{-1}$ must be non-negative and every row of $A$ and $A^{-1}$ must contain at least one non-zero entry.

Suppose there is a row of $A$ with more than one non-zero entry. Since $A$ and $A^{-1}$ may contain no negative numbers, in order to satisfy $A \; A^{-1}=I$ this implies that there are at least two rows of $A^{-1}$ which contain zeroes in the same $m-1$ columns. Such an $A^{-1}$, however, would have a zero determinant and therefore be non-invertible, which is a contradiction.

It follows that each row of $A$ has precisely one positive entry. Since $A$ is invertible it follows that each column of $A$ also has precisely one positive entry so that $A$ is a positively weighted permutation matrix. In terms of the transformation $\mathbf{h}(\mathbf{x})=A\mathbf{x}$ this means the mapping may only positively scale and re-index the components of the vector $\mathbf{x}$, which completes the proof.
\end{proof}

\begin{lemma}
\label{lemma112}
If a mass-action system $\mathcal{N}$ is linearly conjugate to a mass-action system $\mathcal{N}'$ and $\mathcal{N}'$ exhibits locally stable dynamics, then $\mathcal{N}$ exhibits locally stable dynamics.
\end{lemma}

\begin{proof}
The result follows trivially from Lemma \ref{sillylemma} and Definition \ref{conjugate}.
\end{proof}

It is worth noting that other qualitative properties of mass-action systems are also preserved by linear conjugacy (multistability, persistence, boundedness, etc.). Some aspects of qualitative equivalence of $\mathcal{N}$ and $\mathcal{N}'$ can, however, fail for non-linear conjugacies (see the discussion in Section \ref{conclusionsection} and the Appendix).

\subsection{Known Results}
\label{knownresultssection}

In this section, we give a brief summary of the results which are, to the best of our knowledge, the only attempts to demonstrate conjugacy of two mass-action systems. In all of the results described in this section, conjugacy is demonstrated by showing an exact equivalence between the governing differential equations (\ref{de}) for $\mathcal{N}$ and $\mathcal{N}'$. (This phenomenon is called \emph{macro-equivalence} in \cite{H-J1} and \emph{confoundability} in \cite{C-P}. In \cite{Sz2} and the related literature, two networks with identical dynamics are called two \emph{realizations} of the same reaction kinetic differential equations.)

The most thorough study of conjugacy to date has been conducted by Craciun and Pantea \cite{C-P}. In that paper, the authors considered the problem of experimentally assigning values to rate constants to systems with linearly dependent reactions flowing from the same reactant complexes. They present the following result.

\begin{theorem}[Theorem 4.4, \cite{C-P}]
\label{craciun}
There exist rate constants choices such that the mass-action systems $\mathcal{N}$ and $\mathcal{N}'$ are conjugate with $\mathbf{h}(\mathbf{x})=\mathbf{x}$ if and only if they have the same reactant complexes and $C_{\mathcal{R}}(\mathcal{C}^0) \cap C_{\mathcal{R}'}(\mathcal{C}^0) \not= \emptyset$ for every reactant complex $\mathcal{C}^0$, where
\begin{equation}
\label{reactioncone}
C_{\mathcal{R}}(\mathcal{C}^0) = \left\{ \sum_{i=1}^r \alpha_i (\mathbf{z}_i'-\mathbf{z}_i) \; | \; \alpha_i > 0 \mbox{ if } \mathbf{z}_i = \mathbf{z}^0, \alpha_i=0 \mbox{ otherwise} \right\}.
\end{equation}
\end{theorem}

This theorem gives necessary and sufficient conditions for two chemical reaction networks $\mathcal{N}$ and $\mathcal{N}'$ to admit rate constant vectors which generate the same set of governing differential equations (\ref{de}). It is clear that conjugacy follows according to Definition \ref{conjugate}. It should be noted, however, that two reaction networks may be conjugate for some choices of the rate constants and not for others. (It was noted in \cite{Sz1} that this result is deficient in that it ignores the possibility that the net flow from a reactant complex in either $\mathcal{N}$ or $\mathcal{N}'$ could equal zero. Consequently, conjugacy could hold for networks with different reactant complexes so long as the corresponding outflows cancel in (\ref{de}). We will ignore this sensitivity for the time-being.)


In practice, we are typically asked to determine the dynamics of a mass-action system $\mathcal{N}$ without reference to any potential target networks $\mathcal{N}'$. The results of \cite{C-P} propose no methodology to find a potential $\mathcal{N}'$ satisfying Theorem \ref{craciun}, i.e. the target network has to be specified. Szederk\'{e}nyi has extended the work contained in \cite{C-P} by presenting computer algorithms capable of finding networks $\mathcal{N}'$ with equivalent mass-action kinetics (\ref{de}) as $\mathcal{N}$. He calls the networks $\mathcal{N}$ and $\mathcal{N}'$ alternative \emph{realizations} of the kinetics (\ref{de}). In \cite{Sz2} he gives an algorithm for finding sparse and dense realizations (i.e. realizations with the fewest and greatest number of reactions) and in \cite{Sz-H} he and Hangos give an algorithm for determining detailed and complex balanced realizations for specified kinetics (\ref{de}).

Averbukh also considers conditions which relate the dynamics of an undetermined network $\mathcal{N}$ to a network $\mathcal{N}'$ with known dynamics. In particular, he presents conditions under which a general network has the same dynamics as a detailed balanced network (Theorem 2 of \cite{A1}). This is a powerful result since detailed balanced networks are known to exhibit locally stable dynamics \cite{H-J1,V-H}.

Conjugacy is also considered by Krambeck in Section 6 of \cite{K} for detailed balanced systems where it is referred to as \emph{non-uniqueness} of the rate constants. Horn and Jackson briefly consider conjugate systems in their seminal paper \cite{H-J1}. Their primary example is the network
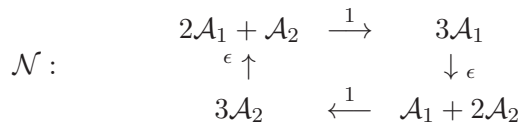
\[\begin{array}{ccccc}
& 2\mathcal{A}_1 + \mathcal{A}_2 & \stackrel{1}{\longrightarrow} & 3\mathcal{A}_1& \\
\mathcal{N}: \hspace{0.5in} & {}^{\epsilon}\uparrow & & \downarrow {}_{\epsilon}& \hspace{0.5in}\\
& 3\mathcal{A}_2 & \stackrel{1}{\longleftarrow} & \mathcal{A}_1 + 2\mathcal{A}_2& 
\end{array}\]
where $\epsilon > 0$. They show that the network exhibits locally stable dynamics for $\epsilon \geq 1/6$ and that the network possesses the same mass-action kinetics (\ref{de}) as a complex balanced network $\mathcal{N}'$ for $\epsilon \geq 1/2$. The Master's thesis of MacLean also contains specific examples of networks which are conjugate to complex balanced systems \cite{M}. This connection with complex balanced systems is made more explicit in her unpublished research notes.

\subsection{Original Results}
\label{originalresultssection}

In this section, we present our main original result regarding conjugacy of chemical reaction networks. In Section \ref{weaklyreversiblenetworkssection}, we show how Theorem \ref{maintheorem} can broaden the scope of weakly reversible networks theory through several illustrative examples.

In what follows, we will let $\mathcal{C}_{react}$ denote the set of reactant complexes in either the complex set $\mathcal{C}$ or the complex set $\mathcal{C}'$.

\begin{theorem}
\label{maintheorem}
Consider two mass-action systems $\mathcal{N}$ and $\mathcal{N}'$. Suppose that for the rate constants $k_i > 0$, $i = 1, \ldots, r$, there exist constants $b_i>0$, $i=1, \ldots, \tilde{r},$ and $c_j>0$, $j=1, \ldots, m$, such that, for every $\mathcal{C}^0 \in \mathcal{C}_{react}$,
\begin{equation}
\label{condition}
\mathop{\sum_{i=1}^r}_{\mathcal{C}_i = \mathcal{C}^0} k_i(\mathbf{z}_i'-\mathbf{z}_i) = T \mathop{\sum_{i = 1}^{\tilde{r}}}_{\tilde{\mathcal{C}}_i = \mathcal{C}^0} b_i(\tilde{\mathbf{z}}_i' - \tilde{\mathbf{z}}_i)
\end{equation}
where $T = $diag$\left\{ c_j \right\}_{j=1}^m$. Then $\mathcal{N}$ is linearly conjugate to $\mathcal{N}'$ with rate constants
\begin{equation}
\label{newrateconstants}
\tilde{k}_i = b_i \prod_{j=1}^m c_j^{\tilde{z}_{ij}}, \hspace{0.5in} i = 1, \ldots \tilde{r}.
\end{equation}
\end{theorem}

It is important to note that the reactant complex set for $\mathcal{C}$ \emph{need not} be the same as that of $\mathcal{C}'$. When $\mathcal{C}^0 \in \mathcal{C}_{react}$ is not an element of the reactant complex set of $\mathcal{C}$, we will consider the summation on the left-hand side of (\ref{condition}) to be empty, and similarly for the right-hand side of (\ref{condition}) when $\mathcal{C}^0$ is not an element of the reactant complex set of $\mathcal{C}'$, i.e.
\[\mathop{\sum_{i=1}^r}_{\mathcal{C}_i = \mathcal{C}^0} k_i(\mathbf{z}_i'-\mathbf{z}_i) = \mathbf{0} \; \; \; \; \; \mbox{and} \; \; \; \; \; \mathop{\sum_{i = 1}^{\tilde{r}}}_{\tilde{\mathcal{C}}_i = \mathcal{C}^0} b_i(\tilde{\mathbf{z}}_i' - \tilde{\mathbf{z}}_i) = \mathbf{0},\]
respectively. In order to satisfy (\ref{condition}), therefore, if one system contains a reactant complex not contained in the other, it is necessary for the origin to lie in the cone generated by the reaction vectors flowing from that reactant complex in the other system.

\begin{proof}
Let $\Phi(\mathbf{x}_0,t)$ correspond to the flow of the mass-action system (\ref{de}) associated to the reaction network $\mathcal{N}$ given by (\ref{crn}). Consider the linear mapping $\mathbf{h}(\mathbf{x})=T^{-1}\mathbf{x}$ where $T = $diag$\left\{ c_j \right\}_{j=1}^m$. Now define $\Psi(\mathbf{y}_0,t)=T^{-1}\Phi(\mathbf{x}_0,t)$ so that $\Phi(\mathbf{x}_0,t) = T \Psi(\mathbf{y}_0,t)$.


Since $\Phi(\mathbf{x}_0,t)$ is a solution of (\ref{de}) for the reaction set (\ref{crn}), we have
\[\begin{split}\Psi'(\mathbf{y}_0,t) & = T^{-1} \Phi'(\mathbf{x}_0,t) \\ & = T^{-1} \sum_{i=1}^r k_i (\mathbf{z}_i'-\mathbf{z}_i) \; \Phi(\mathbf{x}_0,t)^{\mathbf{z}_i} \\ & = T^{-1} \sum_{\mathcal{C}^0 \in \mathcal{C}_{react}} \mathop{\sum_{i=1}^r}_{\mathcal{C}_i = \mathcal{C}^0} k_i (\mathbf{z}_i'-\mathbf{z}_i) \; \Phi(\mathbf{x}_0,t)^{\mathbf{z}^0} \\ & = T^{-1} \sum_{\mathcal{C}^0 \in \mathcal{C}_{react}} T \mathop{\sum_{i=1}^{\tilde{r}}}_{\tilde{\mathcal{C}}_i = \mathcal{C}^0} b_i (\tilde{\mathbf{z}}_i'-\tilde{\mathbf{z}}_i) (T \; \Psi(\mathbf{y}_0,t))^{\mathbf{z}^0} \\ & = \sum_{i=1}^{\tilde{r}} \left( b_i \prod_{j=1}^m c_j^{\tilde{z}_{ij}}\right) (\tilde{\mathbf{z}}_i'-\tilde{\mathbf{z}}_i) \; \Psi(\mathbf{y}_0,t)^{\mathbf{z}_i}.\end{split}\]
It is clear that $\Psi(\mathbf{y}_0,t)$ is the flow of (\ref{de}) for the reaction network (\ref{crn3}) with rate constants given by (\ref{newrateconstants}). We have that $\mathbf{h}(\Phi(\mathbf{x}_0,t))=\Psi(\mathbf{h}(\mathbf{x}_0),t)$ for all $\mathbf{x}_0 \in \mathbb{R}_{>0}^m$ and $t \geq 0$ where $\mathbf{y}_0 = \mathbf{h}(\mathbf{x}_0)$ since $\mathbf{y}_0=\Psi(\mathbf{y}_0,0)=T^{-1}\Phi(\mathbf{x}_0,0)=T^{-1}\mathbf{x}_0$. It follows that the networks $\mathcal{N}$ and $\mathcal{N}'$ are linearly conjugate by Definition \ref{linearconjugate}, and we are done.
\end{proof}

This result gives conditions under which two mass-action systems $\mathcal{N}$ and $\mathcal{N}'$ are linearly conjugate. This is a particularly useful result when the qualitative properties of the original network are obscure while the behaviour of the target network is well understood.

With the exception of the scaling matrix $T$, condition (\ref{condition}) is very similar to the cone intersection condition in Theorem \ref{craciun} where the constants $k_i > 0$, $i =1, \ldots, r,$ and $b_i>0$, $i=1, \ldots, \tilde{r}$, correspond to the magnitudes of the cone generators (i.e. the reaction vectors). If we allow $k_i$ and $b_i$ to vary we have
\[C_{\mathcal{R}}(\mathcal{C}^0) = \left\{ \mathop{\sum_{i=1}^r}_{\mathcal{C}_i = \mathcal{C}^0} k_i(\mathbf{z}_i'-\mathbf{z}_i) \; | \; k_i > 0, i = 1, \ldots, r \right\}\]
and
\[C_{\mathcal{R}'}(\mathcal{C}^0) = \left\{ \mathop{\sum_{i = 1}^{\tilde{r}}}_{\tilde{\mathcal{C}}_i = \mathcal{C}^0} b_i(\tilde{\mathbf{z}}_i' - \tilde{\mathbf{z}}_i) \; | \; b_i > 0, i = 1, \ldots, \tilde{r} \right\}\]
according to (\ref{reactioncone}).

The following two results can be obtained from Theorem \ref{maintheorem} by allowing the rate constant vector $\mathbf{k} \in \mathbb{R}_{>0}^r$ to vary. In these results we let $T=$diag$\left\{ c_j \right\}_{j=1}^m$ and consider $c_j > 0$, $j = 1, \ldots, m$, to be fixed.

\begin{corollary}
\label{corollary1}
Consider two mass-action systems $\mathcal{N}$ and $\mathcal{N}'$. Then there exist rate constant vectors $\mathbf{k} \in \mathbb{R}_{>0}^r$ and $\mathbf{k}' \in \mathbb{R}_{>0}^{\tilde{r}}$ such that $\mathcal{N}$ and $\mathcal{N}'$ are linearly conjugate with $\mathbf{h}(\mathbf{x})=T^{-1}\mathbf{x}$ if and only if for every $\mathcal{C}^0 \in \mathcal{C}_{react}$ we have $C_{\mathcal{R}}(\mathcal{C}^0) \cap [ T \; C_{\mathcal{R}'}(\mathcal{C}^0) ] \not= \emptyset$.
\end{corollary}

\begin{corollary}
\label{corollary2}
Consider two mass-action systems $\mathcal{N}$ and $\mathcal{N}'$. Then for every rate constant vector $\mathbf{k} \in \mathbb{R}_{>0}^r$ there exists a rate constant vector $\mathbf{k}' \in \mathbb{R}_{>0}^{\tilde{r}}$ such that $\mathcal{N}$ is linearly conjugate to $\mathcal{N}'$ with $\mathbf{h}(\mathbf{x}) = T^{-1}\mathbf{x}$ if and only if for every $\mathcal{C}^0 \in \mathcal{C}_{react}$ we have $C_{\mathcal{R}}(\mathcal{C}^0) \subseteq [ T \; C_{\mathcal{R}'}(\mathcal{C}^0) ]$.
\end{corollary}

\begin{proof}
The forward implications follow directly from Theorem \ref{maintheorem}.

To prove the `only if' portions of the results, notice that the assumption of conjugacy with $\mathbf{h}(\mathbf{x})=T^{-1}\mathbf{x}$ implies that
\begin{equation}
\label{9999}
\Psi'(\mathbf{y}_0,t)=\sum_{\mathcal{C}^0 \in \mathcal{C}_{react}} \mathop{\sum_{i=1}^r}_{\mathcal{C}_i = \mathcal{C}^0} T^{-1} \left(k_i \prod_{j=1}^m c_j^{z_{ij}} \right) (\mathbf{z}_i'-\mathbf{z}_i) \; \Psi(\mathbf{y}_0,t)^{\mathbf{z}_0}
\end{equation}
while we have
\begin{equation}
\label{99999}
\Psi'(\mathbf{y}_0,t)=\sum_{\mathcal{C}^0 \in \mathcal{C}_{react}} \mathop{\sum_{i=1}^{\tilde{r}}}_{\mathcal{C}_i = \mathcal{C}^0} \tilde{k}_i (\tilde{\mathbf{z}}_i'-\tilde{\mathbf{z}}_i) \; \Psi(\mathbf{y}_0,t)^{\mathbf{z}_0}
\end{equation}
from (\ref{de}). In order to have equality between (\ref{9999}) and (\ref{99999}) we require that
\[\mathop{\sum_{i=1}^r}_{\mathcal{C}_i = \mathcal{C}^0} \left( k_i \prod_{j=1}^r c_j^{z_{ij}} \right) (\mathbf{z}_i'-\mathbf{z}_i) = T \mathop{\sum_{i=1}^{\tilde{r}}}_{\mathcal{C}_i = \mathcal{C}^0} \tilde{k}_i (\tilde{\mathbf{z}}_i' - \tilde{\mathbf{z}}_i)\]
for every $\mathcal{C}^0 \in \mathcal{C}_{react}$. The desired cone conditions follow immediately from the conditions on the rate constants vectors $\mathbf{k} \in \mathbb{R}_{>0}^r$ and $\mathbf{k}' \in \mathbb{R}_{>0}^{\tilde{r}}$.
\end{proof}

\subsection{Weakly Reversible Networks}
\label{weaklyreversiblenetworkssection}

In Section \ref{conjugatedynamicalsystemssection}, the results depended on having two given networks $\mathcal{N}$ and $\mathcal{N}'$ to compare. In standard practice, however, we have only a single network $\mathcal{N}$ whose dynamics are unknown and we need to find the target network $\mathcal{N}'$ whose dynamics are understood.

In this section, we will consider a particularly broad and well-understood class of such target networks in weakly reversible networks. Since it is known  that weakly reversible systems are complex balanced for at least some values of the rate constants, and therefore exhibit locally stable dynamics for those rate constants values, it is a reasonable starting point when considering a network $\mathcal{N}$ which seems to exhibit locally stable dynamics to search for weakly reversible target networks $\mathcal{N}'$ to which it could be conjugate.

In practice, however, there are many sensitivities which can arise in choosing a suitable target network $\mathcal{N}'$ which is weakly reversible. We will illustrate the applicability, and limitations, of Theorem \ref{maintheorem} to such cases through four examples. The first is an example where linear conjugacy to a weakly reversible network which exhibits locally stable dynamics can be universally shown. The second is an example where linear conjugacy to a weakly reversible network can only be shown for certain choices of the rate vector $\mathbf{k} \in \mathbb{R}_{>0}^r$. This is also an example where $S^*$ and $S$ do not always coincide for the original network $\mathcal{N}$. The third is an example where linear conjugacy to a weakly reversible network holds universally but conditions on $\mathbf{k} \in \mathbb{R}_{>0}^r$ are still required to guarantee locally stable dynamics. This example also demonstrates how these conditions can be reduced by creatively ``splitting'' a reaction in the target network $\mathcal{N}'$. The fourth is an example where a ``phantom'' reactant complex is required to demonstrate linear conjugacy to a weakly reversible network.

Our general technique in this section will be to search for weakly reversible target networks $\mathcal{N}'$ which involve the same reactant complexes as the original network $\mathcal{N}$.\\

\textbf{Example 1:} Consider the chemical reaction network $\mathcal{N}$ given by
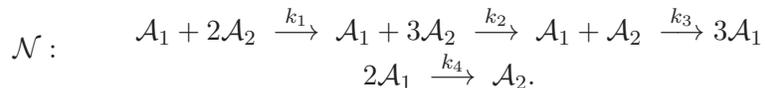
\[\mathcal{N}: \hspace{0.3in} \begin{array}{c} \displaystyle{\mathcal{A}_1 + 2 \mathcal{A}_2 \; \stackrel{k_1}{\longrightarrow} \; \mathcal{A}_1 + 3 \mathcal{A}_2} \; \stackrel{k_2}{\longrightarrow} \; \mathcal{A}_1 + \mathcal{A}_2 \; \stackrel{k_3}{\longrightarrow} 3 \mathcal{A}_1 \\ \displaystyle{2 \mathcal{A}_1 \; \stackrel{k_4}{\longrightarrow} \; \mathcal{A}_2.} \end{array}\]

We want to find a weakly reversible target network $\mathcal{N}'$ involving the reactant complex set of the original network $\mathcal{N}$, which is
\[\left\{ \mathcal{A}_1 + 2 \mathcal{A}_2, \mathcal{A}_1 + 3 \mathcal{A}_2, \mathcal{A}_1 + \mathcal{A}_2, 2 \mathcal{A}_1 \right\}.\]
Many such networks can be eliminated for failing to be weakly reversible, leaving a relatively small set of possibilities. One such possibility is given by
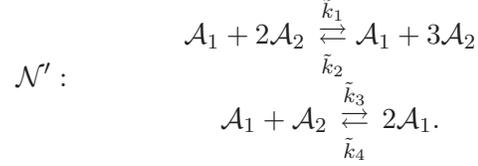
\[\mathcal{N}': \hspace{0.5in}
\begin{array}{c}
\displaystyle{\mathcal{A}_1 + 2 \mathcal{A}_2 \; \mathop{\stackrel{\tilde{k}_1}{\rightleftarrows}}_{\tilde{k}_2}} \; \mathcal{A}_1 + 3 \mathcal{A}_2 \\ \displaystyle{\mathcal{A}_1 + \mathcal{A}_2 \; \mathop{\stackrel{\tilde{k}_3}{\rightleftarrows}}_{\tilde{k}_4} \; 2\mathcal{A}_1.}
\end{array}
\hspace{0.5in}\]
In order for $\mathcal{N}$ and $\mathcal{N}'$ to be conjugate, we need to find $b_i > 0$, $c_j > 0$, $i = 1, \ldots , 4$, $j = 1, 2,$ such that
\[ \begin{split} k_1 \left[ \begin{array}{c} 0 \\ 1 \end{array} \right] & = b_1 \left[ \begin{array}{cc} c_1 & 0 \\ 0 & c_2 \end{array} \right] \left[ \begin{array}{c} 0 \\ 1 \end{array} \right] \\ k_2 \left[ \begin{array}{c} 0 \\ -2 \end{array} \right] & = b_2 \left[ \begin{array}{cc} c_1 & 0 \\ 0 & c_2 \end{array} \right] \left[ \begin{array}{c} 0 \\ -1 \end{array} \right] \\ k_3 \left[ \begin{array}{c} 2 \\ -1 \end{array} \right] & = b_3 \left[ \begin{array}{cc} c_1 & 0 \\ 0 & c_2 \end{array} \right] \left[ \begin{array}{c} 1 \\ -1 \end{array} \right] \\ k_4 \left[ \begin{array}{c} -2 \\ 1 \end{array} \right] & = b_4 \left[ \begin{array}{cc} c_1 & 0 \\ 0 & c_2 \end{array} \right] \left[ \begin{array}{c} -1 \\ 1 \end{array} \right]. \end{split}\]
It can be easily verified that $b_1=k_1$, $b_2=2k_2$, $b_3=k_3$, $b_4=k_4$, $c_1=2$, $c_2=1$ works.

It follows from (\ref{newrateconstants}) that $\mathcal{N}$ is linearly conjugate to $\mathcal{N}'$ with rate constants given by $\tilde{k}_1 = 2k_1$, $\tilde{k}_2 = 4k_2$, $\tilde{k}_3 = 2k_3$ and $\tilde{k}_4 = 4k_4$. It is known that $\mathcal{N}'$ is complex balanced, and therefore exhibits locally stable dynamics, for every rate constant vector $\mathbf{k}' \in \mathbb{R}_{>0}^4$. It follows that the original network $\mathcal{N}$ exhibits locally stable dynamics for every rate constant vector $\mathbf{k} \in \mathbb{R}_{>0}^4$ by Lemma \ref{lemma112}.

It could also be noted that $C_{\mathcal{R}}(\mathcal{C}^0)=T \; C_{\mathcal{R}'}(\mathcal{C}^0)$ for every $\mathcal{C}^0 \in \mathcal{C}_{react}$ so that $\mathcal{N}$ and $\mathcal{N}'$ satisfy the hypotheses of Corollary \ref{corollary2} and therefore linear conjugacy holds unconditionally.\\

\textbf{Example 2:} Consider the chemical reaction network $\mathcal{N}$ given by
\[\mathcal{N}: \hspace{0.5in} \mathcal{A}_2 \; \stackrel{k_1}{\longleftarrow} \; \mathcal{A}_1 \; \stackrel{k_2}{\longleftarrow} \; 2 \mathcal{A}_2 \; \stackrel{k_3}{\longrightarrow}3\mathcal{A}_1. \]

The only weakly reversible target network involving the same reactant complex set as $\mathcal{N}$ is
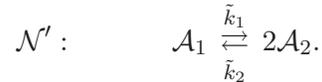
\[\mathcal{N}': \hspace{0.5in} \mathcal{A}_1 \; \mathop{\stackrel{\tilde{k}_1}{\rightleftarrows}}_{\tilde{k}_2} \; 2 \mathcal{A}_2.\]
In order to satisfy (\ref{condition}), we need to find $b_1>0$, $b_2>0$, $c_1>0$, and $c_2>0$ such that
\begin{equation}
\label{25}
\begin{split} & k_1 \left[ \begin{array}{c} -1 \\ 1 \end{array} \right] = b_1 \left[ \begin{array}{cc} c_1 & 0 \\ 0 & c_2 \end{array} \right] \left[ \begin{array}{c} -1 \\ 2 \end{array} \right] \\ & k_2 \left[ \begin{array}{c} 1 \\ -2 \end{array} \right] + k_3 \left[ \begin{array}{c} 3 \\ -2 \end{array} \right] = b_2 \left[ \begin{array}{cc} c_1 & 0 \\ 0 & c_2 \end{array} \right] \left[ \begin{array}{c} 1 \\ -2 \end{array} \right]\end{split}
\end{equation}
while satisfying (\ref{newrateconstants}) requires
\begin{equation}
\label{26}
\tilde{k}_1 = b_1 c_1, \hspace{0.5in} \mbox{and} \hspace{0.5in} \tilde{k}_2 = b_2 c_2^2.
\end{equation}

The system (\ref{25}) corresponds to satisfying $k_1=b_1c_1$, $k_1=2b_1c_2$, $k_2 + 3k_3=b_2c_1$, and $k_2+k_3=b_2c_2$. From the first two equations, we have $c_1/c_2=2$ while from the last two we have $c_1/c_2 = (k_2+3k_3)/(k_2+k_3)$, which implies (\ref{25}) can be satisfied if and only if $k_2=k_3$. With this restriction, the system can be satisfied for $b_1=k_1$, $b_2=2k_2=2k_3$, $c_1=2$, $c_2=1$. It follows from (\ref{26}) that $\tilde{k}_1=2k_1$ and $\tilde{k}_2=2k_2=2k_3$.

It is known that $\mathcal{N}'$ is complex balanced, and therefore exhibits locally stable dynamics, for all values of $\tilde{k}_1>0$ and $\tilde{k}_2>0$; however, because we required a condition on the rate constants of $\mathcal{N}$ in order for condition (\ref{25}) to be satisfied, $\mathcal{N}$ does not exhibit locally stable dynamics unconditionally. In fact, it exhibits locally stable dynamics only for $k_2=k_3$. For $k_2>k_3$, all trajectories tend to the origin, while for $k_3>k_2$ all trajectories become unbounded.

It is worth noting that the kinetic subspace $S^*$ is two-dimensional for $\mathcal{N}$ for all rate constants values except $k_2=k_3$ when it collapses to a single dimension and we have the strict inclusion $S^* \subset S$. Since $\mathcal{N}'$ is weakly reversible, we always have $S^*=S$ for $\mathcal{N}'$ by Lemma \ref{kineticstoichiometric} and we notice that this is always one-dimensional. The systems will only be conjugate when the dimensions of the kinetics compatibility classes match, which only occurs when $k_2=k_3$.

It could also be noted that $C_{\mathcal{R}}(\mathcal{C}^0) \cap [T \; C_{\mathcal{R}'}(\mathcal{C}^0)] \not= \emptyset$ for every $\mathcal{C}^0 \in \mathcal{C}_{react}$ but $C_{\mathcal{R}}(\mathcal{C}^0) \not\subseteq [T \; C_{\mathcal{R}'}(\mathcal{C}^0)]$ for $\mathcal{C}^0 = 2 \mathcal{A}_2 \in \mathcal{C}_{react}$. Consequently, the networks $\mathcal{N}$ and $\mathcal{N}'$ satisfy the hypotheses of Corollary \ref{corollary1} but not Corollary \ref{corollary2}; linear conjugacy with $\mathbf{h}(\mathbf{x})=T^{-1}\mathbf{x}$ cannot therefore be guaranteed for all rate constant vectors $\mathbf{k} \in \mathbb{R}_{>0}^3$.\\

\textbf{Example 3:} Consider the chemical reaction network $\mathcal{N}$ given by
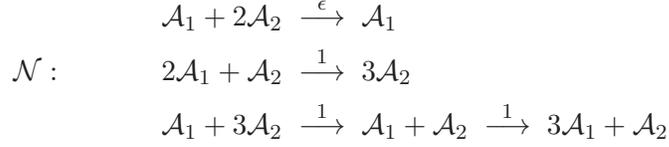
\[\begin{split} \mathcal{A}_1 + 2\mathcal{A}_2 \; & \stackrel{\epsilon}{\longrightarrow} \; \mathcal{A}_1 \\ \mathcal{N}: \hspace{0.5in} 2 \mathcal{A}_1 + \mathcal{A}_2 \; & \stackrel{1}{\longrightarrow} \; 3 \mathcal{A}_2 \\ \mathcal{A}_1 + 3 \mathcal{A}_2 \; & \stackrel{1}{\longrightarrow} \; \mathcal{A}_1 + \mathcal{A}_2 \; \stackrel{1}{\longrightarrow} \; 3 \mathcal{A}_1 + \mathcal{A}_2 \hspace{0.3in} \end{split}\]
for $\epsilon > 0$.

We search for target networks $\mathcal{N}'$ with reactions flowing between the complexes in the reactant set of $\mathcal{N}$, which is
\[\left\{ \mathcal{A}_1 + 2\mathcal{A}_2, 2 \mathcal{A}_1 + \mathcal{A}_2, \mathcal{A}_1 + 3 \mathcal{A}_2, \mathcal{A}_1 + \mathcal{A}_2 \right\}.\]
Many such networks can be eliminated for failing to be weakly reversible, leaving a relatively small set of possibilities. We will choose the network $\mathcal{N}'$ given by
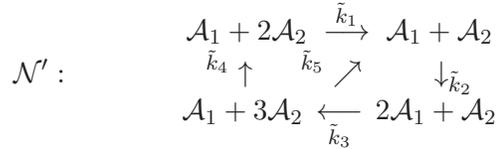
\[\mathcal{N}': \hspace{0.5in}
\begin{array}{c}
\displaystyle{\mathcal{A}_1 + 2 \mathcal{A}_2 \; \stackrel{\tilde{k}_1}{\longrightarrow} \; \mathcal{A}_1 + \mathcal{A}_2} \\ \displaystyle{{}^{\tilde{k}_4} \uparrow \; \; \; \; \; {}^{\tilde{k}_5} \nearrow \; \; \; \; \; \; \; \; \; \downarrow_{\tilde{k}_2}} \\ \displaystyle{\mathcal{A}_1 + 3 \mathcal{A}_2 \; \mathop{\longleftarrow}_{\tilde{k}_3} \; 2\mathcal{A}_1 + \mathcal{A}_2}
\end{array}
\hspace{0.5in}\]
where we have chosen to ``split'' the reaction flowing from the reactant complex $\mathcal{A}_1 + 3 \mathcal{A}_2$ into two weighted reactions. The utility of this technique will become apparent momentarily.

In order to satisfy (\ref{condition}) we need to find constants $b_i>0$, $c_j>0$, $i=1, \ldots, 5$, $j=1,2$, such that
\[\begin{split} \epsilon \left[ \begin{array}{c} 0 \\ -2 \end{array} \right] & = b_1 \left[ \begin{array}{cc} c_1 & 0 \\ 0 & c_2 \end{array} \right] \left[ \begin{array}{c} 0 \\ -1 \end{array} \right] \\ \left[ \begin{array}{c} 2 \\ 0 \end{array} \right] & = b_2 \left[ \begin{array}{cc} c_1 & 0 \\ 0 & c_2 \end{array} \right] \left[ \begin{array}{c} 1 \\ 0 \end{array} \right] \\ \left[ \begin{array}{c} -2 \\ 2 \end{array} \right] & = b_3 \left[ \begin{array}{cc} c_1 & 0 \\ 0 & c_2 \end{array} \right] \left[ \begin{array}{c} -1 \\ 2 \end{array} \right] \\ \left[ \begin{array}{c} 0 \\ -2 \end{array} \right] & = \left[ \begin{array}{cc} c_1 & 0 \\ 0 & c_2 \end{array} \right] \left( b_4 \left[ \begin{array}{c} 0 \\ -1 \end{array} \right] + b_5 \left[ \begin{array}{c} 0 \\ -2 \end{array} \right] \right).\end{split}\]
We will choose the solution set $b_1=2\epsilon$, $b_2=b_3=1$, $b_4=2(1-t)$, $b_5=t$, $c_1=2$, $c_2=1$ where $0 \leq t < 1$ is a weighting constant. This gives rise to the rate constants $\tilde{k}_1 = 4 \epsilon$, $\tilde{k}_2 = 2$, $\tilde{k}_3 = 4$, $\tilde{k}_4 = 4(1-t)$ and $\tilde{k}_5 = 2t$ according to (\ref{newrateconstants}).

There is one condition on the rate constants of $\mathcal{N}'$ in order for the mass-action system to be complex balanced. That condition is
\[\epsilon = \frac{1-t}{\sqrt{2-t}}.\]
Each $0 \leq t < 1$ corresponds to specific network $\mathcal{N}'$ which is conjugate to $\mathcal{N}$. Consequently, we can guarantee that $\mathcal{N}$ is conjugate to a complex balanced system, and therefore exhibits locally stable dynamics, for the range of values $0 < \epsilon \leq 1/\sqrt{2}$. Notice that if we had not split the reaction flowing from the complex $\mathcal{A}_1 + 3 \mathcal{A}_2$ and instead had all of the weight represented in $\tilde{k}_4$, we would have only been able to show that $\mathcal{N}$ exhibits locally stable dynamics for $\epsilon = 1/\sqrt{2}$.

It can be checked empirically that locally stable dynamics appears to be exhibited for $\mathcal{N}$ for all values $\epsilon > 0$. There is a unique positive equilibrium concentration given by $(x_1^*,x_2^*)=(1,-\epsilon/2+\sqrt{(\epsilon/2)^2+1})$ which is locally asymptotically stable for all $\epsilon > 0$ according to standard linearization theory. It is not our claim, therefore, that this theory represents a complete classification of locally stable dynamics, even for networks which exhibit locally stable dynamics for some values of the rate constants.\\

\textbf{Example 4:} Consider the chemical reaction network $\mathcal{N}$ given by
\[\mathcal{A}_1 \; \stackrel{k_1}{\longrightarrow} \; 2\mathcal{A}_1 + 2\mathcal{A}_2 \; \stackrel{k_2}{\longrightarrow} \; \mathcal{A}_2 \; \stackrel{k_3}{\longrightarrow} \; \mathcal{A}_1 + \mathcal{A}_2.\]

A quick analysis (\ref{de}) reveals that $\mathcal{N}$ appears to exhibit locally stable dynamics which suggests that the network may be linearly conjugate to a complex balanced system. It can be quickly seen, however, that there is no weakly reversible network $\mathcal{N}'$ involving the same reactant complexes as $\mathcal{N}$ which satisfies the requirements of Theorem \ref{maintheorem}.

Theorem \ref{maintheorem}, however, does not require the target network $\mathcal{N}'$ to use the same reactant complexes as $\mathcal{N}$. We can have a reactant complex $\mathcal{C}^0$ from the reactant complex set of $\mathcal{C}'$ which is not in the reactant complex set of $\mathcal{C}$ so long as
\[\mathop{\sum_{i = 1}^{\tilde{r}}}_{\tilde{\mathcal{C}}_i = \mathcal{C}^0} b_i(\tilde{\mathbf{z}}_i' - \tilde{\mathbf{z}}_i) = \mathbf{0}.\]
One possible target network $\mathcal{N}'$ which makes use of such a ``phantom'' reactant complex is given by
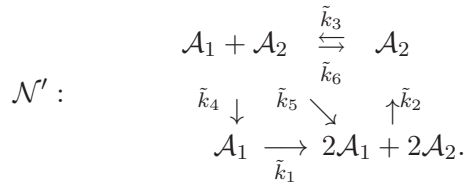
\[\mathcal{N}': \hspace{0.5in}
\begin{array}{l}
\displaystyle{\mathcal{A}_1 + \mathcal{A}_2 \; \; \; \mathop{\stackrel{\tilde{k}_3}{\leftrightarrows}}_{\tilde{k}_6} \; \; \; \mathcal{A}_2} \\ \; \; \displaystyle{{}^{\tilde{k}_4} \downarrow \; \; \; {}^{\tilde{k}_5} \searrow \; \; \; \; \; \; \uparrow^{\tilde{k}_2}} \\ \; \; \; \;\displaystyle{\mathcal{A}_1 \; \mathop{\longrightarrow}_{\tilde{k}_1} \; 2\mathcal{A}_1 + 2\mathcal{A}_2.}
\end{array}
\hspace{0.5in}\]
Since the reactions corresponding to $\tilde{k}_1$, $\tilde{k}_2$ and $\tilde{k}_3$ in $\mathcal{N}$ are the same as those for $k_1$, $k_2$, and $k_3$ in $\mathcal{N}$, we set $c_1=c_2=1$, $b_1=k_1$, $b_2=k_2$, and $b_3=k_3$. Since the complex $\mathcal{A}_1 + \mathcal{A}_2$ is not a reactant complex in $\mathcal{C}$, in order to satisfy (\ref{condition}) it is required that
\[\tilde{k}_4 \left[ \begin{array}{c} 0 \\ -1 \end{array} \right] + \tilde{k}_5 \left[ \begin{array}{c} 1 \\ 1 \end{array} \right] + \tilde{k}_6 \left[ \begin{array}{c} -1 \\ 0 \end{array} \right] = \left[ \begin{array}{c} 0 \\ 0 \end{array} \right].\]
This can be satisfied if $\tilde{k}_4=\tilde{k}_5=\tilde{k}_6=t$ for any $t > 0$.

There is one condition on the rate constants of $\mathcal{N}'$ required for the network to be complex balanced. That condition in terms of the rate constants of $\mathcal{N}$ is
\[6t^3 = k_1k_2k_3.\]
Since every value of $t > 0$ corresponds to a valid conjugate network $\mathcal{N}'$, this is no restriction at all. It follows that the original network $\mathcal{N}$ is linearly conjugate to a complex balanced system for all choices of rate constants and therefore universally exhibits locally stable dynamics by Lemma \ref{lemma112}.

In other words, we are able to demonstrate the network $\mathcal{N}$ is linearly conjugate to a complex balanced network $\mathcal{N}'$, and therefore possesses locally stable dynamics, by adding a ``phantom'' reactant complex which contributes no dynamical information to the mass-action kinetics (\ref{de}). It should be noted, however, that in order for the target network $\mathcal{N}'$ satisfying (\ref{condition}) to be weakly reversible it is necessary that any such phantom reactant complex at least appear in the set of product reactants of $\mathcal{N}$.

\section{Conclusion}
\label{conclusionsection}

In this paper we have summarized what are, to the best of our knowledge, the most comprehensive results on the topic of chemical reaction networks with qualitatively equivalent dynamics under the assumption of mass-action kinetics. We have attempted to unify these results into a common theory which we have called conjugate chemical reaction network theory. We have presented a result (Theorem \ref{maintheorem}) which goes beyond the results currently in the literature and shown in Section \ref{weaklyreversiblenetworkssection} how this result can be used to expand the scope of weakly reversible networks theory.

There are several immediate avenues for future work in the field of conjugate chemical reaction networks.
\begin{enumerate}

\item
It is an open question the extent to which \emph{non-linear} mappings $\mathbf{h}: \mathbb{R}_{>0}^m \mapsto \mathbb{R}_{>0}^m$ can offer insight into the behaviour of chemical reaction networks. In order to preserve mass-action dynamics, it is clear that we should select $\mathbf{h}(\mathbf{x})$ from the class of polynomial transformations. The algebra, however, is non-trivial except for simple cases, and it is known that non-linear mappings need not preserve the form of the reactant complexes or all aspects of the qualitative behaviour of the mass-action system (see Appendix).

\item
The primary focus of Section \ref{weaklyreversiblenetworkssection} was determining when a reaction network $\mathcal{N}$ could be shown to be conjugate to a weakly reversible network $\mathcal{N}'$. For all but the smallest examples, however, it is typically not feasible to check all of the conditions required for conjugacy over the set of possible weakly reversible target networks by hand. The development of algorithms and computer software which can efficiently check for viable weakly reversible target networks $\mathcal{N}'$, and which makes full use of the techniques of ``splitting'' reactions and using ``phantom'' reactions, is a primary interest.

We note that Szederk\'{e}nyi and Hangos have presented some computer algorithms for the related problem of determining when a network has mass-action kinetics (\ref{de}) equivalent to that of a complex balanced system for a specified set of rate constants \cite{Sz-H}.

\item
We have assumed throughout this paper that a chemical reaction network $\mathcal{N}$ obeys the mass-action kinetics (\ref{de}), which corresponds to
\[\frac{d\mathbf{x}}{dt} = \sum_{i=1}^r (\mathbf{z}_i'-\mathbf{z}_i) \; R_i(\mathbf{x})\]
with $R_i(\mathbf{x}) = k_i \prod_{j=1}^m x_j^{z_{ij}}$. Within the literature, and in particular within the systems biology literature, one often encounters other kinetics. A popular choice for an enzyme-catalysed single-substrate reaction is the Michaelis-Menten kinetics
\[R_i(\mathbf{x}) = \frac{V_{max} x}{K_M + x}\]
where $x$ is the concentration of the substrate. Another popular choice is a Hill function \cite{Hi}. Some authors consider general properties on the $R_i(\mathbf{x})$, such as convexity, strict monotonicity with respect to reactant concentrations, etc. \cite{A3,S}. It remains an open question as to how conjugacy can be applied to these alternate kinetic schemes and which aspects of qualitative dynamics are preserved.
\end{enumerate}

\begin{center}
\textbf{\large Appendix}
\end{center}
\normalsize

Comparing the qualitative behaviour of two conjugate mass-action systems is more challenging when the diffeomorphism $\mathbf{h}: \mathbb{R}_{>0}^m \mapsto \mathbb{R}_{>0}^m$ is non-linear. In general, many qualitative properties are not preserved by non-linear mappings.

For example, consider the mass-action networks $\mathcal{N}$ and $\mathcal{N}'$ given by
\[\mathcal{N}: \hspace{0.5in} \mathcal{A} \; \stackrel{k}{\longrightarrow} \; \mathcal{O}\]
and
\[\mathcal{N}': \hspace{0.5in} \mathcal{A} \; \stackrel{k}{\longrightarrow} \; 2\mathcal{A}.\]
Solutions of (\ref{de}) associated with $\mathcal{N}$ are given by $x(t)=x_0e^{-kt}$ while solutions of (\ref{de}) for $\mathcal{N}'$ are given by $y(t)=y_0e^{kt}$ which can be mapped into one another via the non-linear mapping $y(t)=1/x(t)$. It is clear the qualitative behaviour of $\mathcal{N}$ is not preserved through the conjugacy mapping.

More complicated disparities in behaviour can occur in higher dimensions. For instance, the reaction network $\mathcal{N}$ given by
\[\mathcal{N}: \hspace{0.5in} \mathcal{A}_1 \; \stackrel{k}{\longrightarrow} \; \mathcal{A}_2\]
is conjugate to the reaction network $\mathcal{N}'$ given by
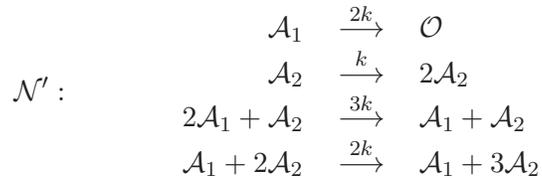
\[\mathcal{N}': \hspace{0.5in} \begin{array}{rcl} \displaystyle{\mathcal{A}_1} & \displaystyle{\; \stackrel{2k}{\longrightarrow} \;} & \displaystyle{\mathcal{O}} \\ \displaystyle{\mathcal{A}_2} & \displaystyle{\; \stackrel{k}{\longrightarrow} \;} & \displaystyle{2\mathcal{A}_2} \\ \displaystyle{2\mathcal{A}_1 + \mathcal{A}_2} & \displaystyle{\; \stackrel{3k}{\longrightarrow} \;} & \displaystyle{\mathcal{A}_1 + \mathcal{A}_2} \\ \displaystyle{\mathcal{A}_1 + 2\mathcal{A}_2} & \displaystyle{\; \stackrel{2k}{\longrightarrow} \;} & \displaystyle{\mathcal{A}_1 + 3\mathcal{A}_2} \end{array}\]
via the mapping
\[\begin{split} y_1 & = x_1^2 \; x_2^{-3} \\ y_2 & = x_1^{-1} \; x_2^2.\end{split}\]
In this case, we have the limiting behaviour $x_1(t) \to 0$ and $x_2(t) \to x_1(0)+x_2(0)$ for $\mathcal{N}$ but $y_1(t) \to 0$ and $y_2(t) \to \infty$ for $\mathcal{N}'$. The qualitative behaviour is significantly different.

There are still examples, however, for which non-linear conjugacies are capable of demonstrating qualitative equivalence of solutions. For example, consider the network
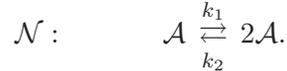
\[\mathcal{N}: \hspace{0.5in} \mathcal{A} \; \mathop{\stackrel{k_1}{\rightleftarrows}}_{k_2} \; 2\mathcal{A}.\]
It is easy to see that $\mathcal{N}$ is conjugate to the reaction networks $\mathcal{N}_b$ given by
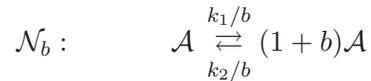
\[\mathcal{N}_b: \hspace{0.5in} \mathcal{A} \; \mathop{\stackrel{k_1/b}{\rightleftarrows}}_{k_2/b} \; (1+b)\mathcal{A}\]
for $b \in \mathbb{Z}_{>0}$ via the diffeomorphism $y(t)=x(t)^{1/b}$ where $y(t)$ is the solution of the mass-action system (\ref{de}) for $\mathcal{N}_b$. The set of reactant complexes changes in all cases except $b=1$; however, the qualitative behaviour of the systems are identical.

\end{document}